\documentclass[a4paper,11pt]{article}
\usepackage{amsfonts}
\usepackage{amsmath}
\usepackage{amsthm}
\usepackage{amssymb}
\usepackage{algorithm}
\usepackage[noend]{algpseudocode}
\usepackage[hyphens]{url}
\usepackage{hyperref}
\usepackage{graphicx}
\usepackage{subfigure}
\usepackage{xspace}
\usepackage{units}

\newtheorem{theorem}{Theorem}

\newtheorem{proposition}[theorem]{Proposition}
\newtheorem{lemma}[theorem]{Lemma}

\begin{document}
\title{Simulating a die roll by flipping two coins}
\author{Giovanni Viglietta\thanks{School of Electrical Engineering and Computer Science, University of Ottawa, Ottawa ON, Canada, \protect\url{viglietta@gmail.com}.}}

\maketitle

\begin{abstract}
We show how to simulate a roll of a fair $n$-sided die by one flip of a biased coin with probability $1/n$ of coming up heads, followed by $3\lfloor\log_2 n \rfloor+1$ flips of a fair coin.
\end{abstract}

Let a \emph{$p$-coin} be a biased coin with probability $p$ of coming up heads. A \emph{fair coin} is a $(1/2)$-coin. A \emph{fair $n$-sided die} is a die with $n$ sides, each of which comes up with probability $1/n$ when the die is rolled.

\begin{proposition}\label{p:1}
For every positive integer $n$, and $k=\lfloor\log_2 n \rfloor$, there are two integers $a$ and $b$, with $0\leqslant a,b\leqslant 2^{k+1}$, such that
$a+b(n-1)=2^{2k+1}.$
\end{proposition}
\begin{proof}
If $n=2^k$, we let $a=b=2^{k+1}$. Otherwise $2^k< n< 2^{k+1}$, and we let $a$ and $b$ be, respectively, the remainder and the quotient of the Euclidean division of $2^{2k+1}$ by $n-1$. Hence $0\leqslant a<n-1<2^{k+1}$, $b\geqslant 0$, and $$b=\frac{2^{2k+1}-a}{n-1}\leqslant \frac{2^{2k+1}}{2^k}=2^{k+1}.$$
\end{proof}

\begin{lemma}\label{l:1}
A flip of a $(2^k/n)$-coin, where $k=\lfloor\log_2 n \rfloor$, can be simulated by one flip of a $(1/n)$-coin and $k+1$ flips of a fair coin.
\end{lemma}
\begin{proof}
We show an algorithm that, given the outcomes of the coin flips, outputs ``heads'' with probability $2^k/n$ and ``tails'' with probability $1-2^k/n$. Let $a$ and $b$ be as in Proposition~\ref{p:1}, and let $d<2^{k+1}$ be the non-negative integer whose $i$-th binary digit is $1$ if and only if the fair coin comes up heads on the $i$-th flip, with $1\leqslant i\leqslant k+1$. We output ``heads'' if and only if the $(1/n)$-coin comes up heads and $d<a$, or the $(1/n)$-coin comes up tails and $d<b$. This happens with probability $$\frac 1n\cdot\frac a{2^{k+1}}+\frac{n-1}n\cdot\frac b{2^{k+1}}=\frac{2^{2k+1}}{n\cdot 2^{k+1}}=\frac{2^k}n.$$
\end{proof}

\begin{lemma}\label{l:2}
A roll of a fair $n$-sided die can be simulated by one flip of a $(2^k/n)$-coin and $2k$ flips of a fair coin, where $k=\lfloor\log_2 n \rfloor$.
\end{lemma}
\begin{proof}
We show an algorithm that, given the outcomes of the coin flips, outputs an integer in $[0,n-1]$ with uniform probability. Let $m=n-2^k$, and let $d<2^k$ (respectively, $d'<2^k$) be the non-negative integer whose $i$-th binary digit is $1$ if and only if the fair coin comes up heads on the $i$-th flip (respectively, the $(k+i)$-th flip), with $1\leqslant i\leqslant k$. If the $(2^k/n)$-coin comes up tails, we output $d$. If the $(2^k/n)$-coin comes up heads and $d\geqslant m$, we output $d'$. If the $(2^k/n)$-coin comes up heads and $d<m$, we output $2^k+d$. Therefore, we output each integer in $[0,2^k-1]$ with probability $$\frac mn\cdot\frac 1{2^k}+\frac{2^k}n\cdot\frac{2^k-m}{2^k}\cdot\frac 1{2^k}=\frac 1n,$$ and we output each integer in $[2^k,n-1]$ with probability $$\frac{2^k}n\cdot\frac{m}{2^k}\cdot\frac 1m=\frac 1n.$$
\end{proof}

\begin{theorem}
A roll of a fair $n$-sided die can be simulated by one flip of a $(1/n)$-coin and $3\lfloor\log_2 n \rfloor+1$ flips of a fair coin.
\end{theorem}
\begin{proof}
Let $k=\lfloor\log_2 n \rfloor$. By Lemma~\ref{l:1} we can simulate a flip of a $(2^k/n)$-coin by one flip of the $(1/n)$-coin and $k+1$ flips of the fair coin. Then, by Lemma~\ref{l:2}, we can simulate a roll of a fair $n$-sided die by further flipping the fair coin $2k$ times.
\end{proof}

{\small \paragraph{\small Acknowledgments.} The author wishes to thank Daniele Bonotto, Alessandro Cobbe, Marcello Mamino, and Federico Poloni for insightful discussions.}

\small
\bibliographystyle{abbrv}

\end{document}